\documentclass[11pt,twoside]{amsart}
\usepackage{latexsym,amsfonts,amsmath,amssymb,soul}
\usepackage{enumerate,soul}
\usepackage[usenames, dvipsnames]{color}
\numberwithin{equation}{section}

\makeatletter
\@namedef{subjclassname@2020}{%
  \textup{2020} Mathematics Subject Classification}
\makeatother

\newtheorem{Th}{Theorem}[section]
\newtheorem{Rem}[Th]{Remark}

\newtheorem{Cor}[Th]{Corollary}

\catcode`\@=11

\renewcommand{\section}%
   {\setcounter{equation}{0}\@startsection {section}{1}{\z@}{-3.5ex plus -1ex
  minus -.2ex}{2.3ex plus .2ex}{\Large\bf}}

\def\ds{\displaystyle}
\def\R{\mathbb R}

\def\T{\mathbb T}

\newcommand{\beqsn}{\arraycolsep1.5pt\begin{eqnarray*}}
\newcommand{\eeqsn}{\end{eqnarray*}\arraycolsep5pt}
\newcommand{\beqs}{\arraycolsep1.5pt\begin{eqnarray}}
\newcommand{\eeqs}{\end{eqnarray}\arraycolsep5pt}

\title{Notes on a paper of Pokhozhaev}

\author[Boiti]{Chiara Boiti}
\address{
Dipartimento di Matematica e Informatica \\Universit\`a di Ferrara\\
Via Ma\-chia\-vel\-li n.~30\\
I-44121 Ferrara\\
Italy}
\email{chiara.boiti@unife.it}

\author[Manfrin]{Renato Manfrin}
\address{Dipartimento di Culture del Progetto\\
Universit\`a IUAV di Venezia\\
Dorsoduro 2196, Cotonificio Veneziano\\
30123 Venezia\\
Italy}
\email{manfrin@iuav.it}

\begin{document}


\begin{abstract}
We prove  a second order  identity for the  Kirchhoff equation
which yields, in particular, a simple and direct proof
of Pokhozhaev's second order conservation law when the nonlinearity has the special form
$(C_1 s +C_2)^{-2}$. As applications, we give: an estimate of order $\varepsilon^{-4}$ for the lifespan $T_\varepsilon$ of the solution of the Cauchy problem with initial data of size $\varepsilon$ in Sobolev spaces when the nonlinearity is given by any $C^2$ function $m(s)>0$; a  necessary and sufficient condition for boundedness of a second order energy of the solutions.
\end{abstract}

\maketitle

\markboth{\sc  Notes on a paper of Pokhozhaev}
 {\sc C.~Boiti, R.~Manfrin}

\section{Introduction}
 As it was proved by
S.\,I.~Pokhozhaev  (\cite{P2}, \cite{P}),  if the nonlinearity has the  form
\beqsn
a(s)= \frac {1}{(C_1 s + C_2)^2} \quad\; \text{($C_1, C_2 \in \R$ not both zero)}
\eeqsn
and if $u$ is a sufficiently regular  solution in $\Omega\times [0,T)$ ($\Omega \subset \R^n$ a bounded, $C^2$ domain; $T > 0$) of
 the Kirchhoff  equation
\beqs
\label{1}
u_{tt}-a \left (\int_\Omega|\nabla u|^2dx \right )\Delta u=0, \quad u=0 \;\, \text{on}\;\, \partial \Omega \times [0,T),
\eeqs
then the second order functional
\beqs \label{PoII}
{  I}u(t)  :=   (C_1 s +C_2  )\int_\Omega |\nabla u_t|^2\, dx + \frac {1}
 {C_1 s +C_2 } \int_\Omega |\Delta u |^2\, dx - C_1 
\left (\int_\Omega \nabla u \!\cdot\!  \nabla u_t  \,dx\right )^2,
\eeqs
with
\beqsn
s = s(t)=
 \int_\Omega |\nabla u(x,t)|^2 \,dx,
 \eeqsn
remains constant in $[0,T)$.\footnote{\,It is clearly understood that we are assuming 
$ C_1 s(t)+ C_2 \ne 0$ in $[0,T)$.}
See \cite[formula (2.3)]{P}.

This is the so-called Pokhozhaev's second order conservation law, associated
 with  \eqref{1}, whose proof was only sketched in \cite{P}, referring to
\cite{P2} (in Russian) for more details. See also \cite{MP} for a simpler explanation and \cite{PN} for some applications to
the global existence of low regularity solutions of \eqref{1}. We show here that this conservation law is a consequence of a second order identity, formula \eqref{4} below, which holds for every $C^2$ nonlinearity
 $m$, as long as
$m(s) > 0$. Despite its elementary nature, this identity does not seem
to be well known.
We also give two applications: (1) an estimate of order $\varepsilon^{-4}$ for the life span
$T_\varepsilon$ of the solution of the Cauchy problem with initial data of size $\varepsilon$ in
Sobolev spaces, for every $C^2$ function $m(s)>0$, extending the result of
\cite{BH} where $m(s)=1+s$ and $\Omega=\T^n$.
See also \cite{BH2} where an estimate of order $\varepsilon^{-6}$ for $T_\varepsilon$
is given, but assuming in addition a rather strong condition on the initial data;
(2) a necessary and sufficient condition for boundedness (or possible
blow-up) of the  second order energy $E(t)$; see \eqref{5} and  Theorem~\ref{th5} below.

\section{A second order identity for general nonlinearity}
For simplicity from now on we will assume  $\Omega = \R^n$. If $\Omega  \subset\R^n$ is a bounded, $C^2$ domain or if $\Omega$ is the n-dimensional torus $\T^n $, our arguments still work by adopting the usual boundary conditions.

So let us suppose $\Omega = \R^n$ and
\beqs
\label{rg}
u\in C^k ([0,T) ; H^{2-k}(\R^n)), \quad k=0,1,2,
\eeqs
 a given solution in $\R^n \times [0,T)$, $ T > 0$, of  the equation
\beqs
\label{1bis}
u_{tt}-m \left (\int|\nabla u|^2dx \right )\Delta u=0,
\eeqs
with $m:\,J  \rightarrow \R$ ($J \subset \R$  open interval)  such that
\beqs
\label{pos}
m \in C^2 (J), \quad m(s) >0 \quad \text{for}\   s\in J, \quad
 \int |\nabla u |^2 dx  \in J \quad\text{for}\  t \in [0,T).
\eeqs
Taking into account \eqref{rg} and \eqref{pos}, we may define the {\em  energy}
\beqs
\label{5}
E(t):=\frac{1}{\sqrt{m}}\int |\nabla u_t|^2dx+\sqrt{m}\int |\Delta u|^2dx,
\eeqs
with $m=m(s)$ and
\beqsn
s=s(t):= \int |\nabla u |^2 dx \quad \text{for} \quad t\in [0,T).
\eeqsn

\subsection{Second order  identity} Under the previous assumptions, we can establish the following:
\beqs
\label{4}
\frac{d}{dt}\left[E(t)- \frac{1}{4}  \frac {d}{ds} \left ( \frac {1}{\sqrt {m}} \right)  s'(t)^2\right]= -
\frac{1}{4} \frac {d^2}{ds^2} \left ( \frac {1}{\sqrt {m}} \right)   s'(t)^3,
\eeqs
where
\beqsn
\frac {d^k}{ds^k} \left ( \frac {1}{\sqrt {m}} \right) =
 \frac {d^k}{ds^k} \left ( \frac {1}{\sqrt {m}} \right) (s) \quad \text{for}\;\,  \quad k =1,2.
 \eeqsn
\begin{proof} From the assumption \eqref{rg} it easily follows that $s$ is a $C^2$
 function in $[0,T)$.
In particular,
\beqs \label{es1}
s'(t)=2\int \nabla u \!\cdot \! \nabla u_t  \, dx
\eeqs
and, by \eqref{1bis},
\beqs 
\label{es2}
s''(t) =
2\int|\nabla u_t|^2\, dx-2\,m  \int |\Delta u|^2\, dx.
\eeqs
Now, we merely need to derive  $E(t)$. The following  calculation makes sense for a sufficiently regular solution,
but it can be justified by  density arguments (see, for instance,  \cite [p.\,23]{Y}). We find:
\beqs
\nonumber
\ \ E'(t)=&&\left(\frac{1}{\sqrt{m}}\right)'\int |\nabla u_t|^2dx  +   \big (\sqrt{m} \big )' \int |\Delta u|^2dx+ \frac{2}{\sqrt{m}} \int \big [
\nabla u_{t}\nabla u_{tt}
 +  {m} \Delta u\Delta u_t \big  ]dx
 \\
 \nonumber
=&&\frac{d}{ds}\left ( \frac{1}{\sqrt{m}} \right ) s'(t)\int |\nabla u_t|^2dx +
 \frac{d}{ds} \big ({\sqrt{m}} \big )  s'(t)\int |\Delta u|^2 dx  -
\frac{2}{\sqrt{m}}\int \big  [ u_{tt}-m \Delta u \big ]\Delta u_t  dx\\
\label{stella}
=&& \frac{d}{ds} \left ( \frac{1}{\sqrt{m}} \right ) s'(t)  \left[  \int |\nabla u_t|^2dx-m \int
|\Delta u|^2dx  \right],
\eeqs
which immediately says that  $E(t) $ is a $C^1$ function in $[0,T)$. Taking into account \eqref{es2}, we have
\beqsn
E'(t)= \frac{d}{ds}  \left ( \frac{1}{\sqrt{m}} \right ) s'(t) \frac {s''(t)}{2} =
\frac{1}{4}  \frac{d}{ds} \left ( \frac{1}{\sqrt{m}} \right ) \frac {d}{dt} s'(t)^2,
\eeqsn
 and then the identity  \eqref{4}.
 \end{proof}
 
\section {A simple proof of Pokhozhaev's second order conservation law}
We first note that if $m\in C^2$ and $ m (s)> 0$ in the open interval $J\subset \R$, then
\beqsn
\frac {d^2}{ds^2} \left ( \frac {1}{\sqrt {m}} \right) \equiv 0 \; \; \text{in} \; \, J  \quad \Leftrightarrow \quad
\frac {1}{ \sqrt{m}} = C_1 s + C_2 \; \; \text{in} \; \, J ,
 \eeqsn
with $C_1, C_2\in \R$ such that $C_1s + C_2 > 0$ for $s\in J$.  This means, in particular, that
  \eqref{4} gives
\beqs
\label{4bis}
E(t)- \frac{1}{4} \frac {d}{ds} \left ( \frac {1}{\sqrt {m}} \right)  s'(t)^2 = \text{const.}
 \quad \quad \forall t \in [0,T),
\eeqs
if  $m$ has the special  form $m(s)=(C_1s+C_2)^{-2}$.
\\
Then, writing explicitly \eqref{4bis} for this particular nonlinearity,
we immediately obtain {\em Pokhozhaev's
second  order conservation law}, i.e.,
\beqsn
Iu(t)  =   \text{const.}   \quad \quad \forall t \in [0,T),
\eeqsn
where $ Iu $ is the second order functional \eqref{PoII}, with $\Omega = \R^n$.

\section{An estimate of the life-span for small initial data}
We consider now the Cauchy problem  with initial data of size $\varepsilon$. That is,
\beqs
\label{CP}
\begin{cases}
\ds u_{tt}-m \left(\int |\nabla u|^2dx\right)\Delta u=0,&(x,t)\in \R^n \times\R^+_t\cr
\ds u(x,0)=\varepsilon u_0(x), \quad
\ds u_t(x,0)=\varepsilon   u_1(x)
\end{cases}
\eeqs
for  fixed  $u_0,u_1$ (not both zero) and $\varepsilon > 0$ small enough
(here $\R^+_t=\{t\geq0\}$).

\begin{Th}
\label{th}
Let $m\in C^2(\R^+)$, $m(s) >0$ for all $s \in \R^+$, and let us suppose
\beqsn
u_0 \in H^2(\R^n) , \quad u_1\in H^1(\R^n).
\eeqsn

There exist then constants $C,\delta > 0$  such that
the Cauchy problem \eqref{CP} has a unique solution $
u\in C^k([0,T_\varepsilon );H^{2-k}(\R^n))$ $(k=0,1,2)$ with
\beqsn
T_\varepsilon >  \frac {C}{\varepsilon^{4}}, \quad \text{if} \quad 0 < \varepsilon \leq \delta  .
\eeqsn
\end{Th}

\begin{proof}
Let us first recall the well known  {\em first order conservation law}
\beqs
\label{62}
\int |u_t|^2dx+M\left(\int |\nabla u|^2dx\right)
=\varepsilon^2\int |u_1|^2dx+M\left(\varepsilon^2\int |\nabla u_0|^2dx\right),
\eeqs
independently of $t\ge 0$, and where $M:  \R^+ \rightarrow \R^+ $ is the continuous, strictly increasing function
\beqs
\label{EM}
M(s)=\int_0^sm(h) dh .
\eeqs

It is well known that \eqref{62},
under the assumption that $m$ is coercive at $\infty$
(i.e. $\int_0^\infty m(s) ds = + \infty$), implies that the Cauchy problem \eqref{CP}, with initial data
$u(0,x)\in H^2(\R^n)$ and $u_t(0,x)\in H^1(\R^n)$, 
has a unique local solution $u\in C^k([0,T);H^{2-k}(\R^n))$ $(k=0,1,2)$ for some $T > 0$.
See the general result of \cite[Theorem~2.1]{AP} and the bibliography therein.
 See also \cite{Y}.\footnote{\,Theorem~I of \cite{Y} proves that the Cauchy problem for the dissipative Kirchhoff equation is locally well posed. But the proof
can be easily adapted to the non-dissipative case \eqref{CP}.}

Coercivity at $\infty$ is used to prove the boundedness of $\int|\nabla u|^2dx$.
Here the assumption of small initial data allows us to merely assume $m(s)>0$ because, in this case, the
first order conservation law \eqref{62} is enough to obtain that $ \int|\nabla u|^2dx$ remains small.
This, in turn, implies that $m(\int|\nabla u|^2dx)$ remains greater than a positive constant (see \eqref{bound0}-\eqref{summ} below).
Let us also remark that \cite{AP} would give an estimate of the lifespan $T_\varepsilon$ of the solution of order $\varepsilon^{-2}$, while here the second order identity \eqref{4} allows
us to easily get an estimate of $T_\varepsilon$ of order $\varepsilon^{-4}$.

To prove it let us set
\beqsn
\quad N({\varepsilon}) := \varepsilon^2 \int |u_1|^2dx+M\left(\varepsilon^2 \int |\nabla u_0|^2dx\right).
\eeqsn
From \eqref{62} it is clear that $ \int |u_t|^2dx $ is bounded above by $N(\varepsilon)$.
To obtain an {\it a-priori} bound for $\int |\nabla u |^2dx $
let us first remark that $\lim_{\varepsilon\to0^+}N(\varepsilon)=0$ and hence there exists
$\varepsilon_0>0$ such that
\beqsn
N(\varepsilon)\leq N(\varepsilon_0)<\int_0^{+\infty}\!\!m(h)dh,
\qquad\forall\, 0\leq\varepsilon\leq\varepsilon_0.
\eeqsn
This implies that $M^{-1}(N(\varepsilon))$ is well defined for all $0\leq \varepsilon\leq \varepsilon_0$ and
\beqs
\label{bound0}
s(t)=\int |\nabla u|^2dx \le M^{-1}(N(\varepsilon) ) \le M^{-1}(N(\varepsilon_0) ) < \infty,
\qquad\forall t\geq0,\ 0\leq\varepsilon\leq\varepsilon_0.
\eeqs
Since
 \beqs
 \label{49}
 0< m_0 \le  m(s)  \le  m_1  \quad \text{if} \quad 0 \le s \le M^{-1}(N(\varepsilon_0))
 \eeqs
for suitable  constants $m_0,m_1$, it easily follows from
\eqref{62} that there exists $c_0>0$ such that
\beqs
\label{summ}
s(t)=\int |\nabla u|^2dx  \le   c_0 \varepsilon^2,
\qquad\forall t\geq0,\ 0\leq\varepsilon\leq\varepsilon_0.
\eeqs

From \eqref{es1}, H\"older's inequality, \eqref{summ} and \eqref{5}, we thus obtain, for  $0\leq\varepsilon\leq \varepsilon_0$:
\beqs \label{11}
\begin{aligned}
s'(t)^2 &=  4\left(\int\nabla u\cdot\nabla u_tdx\right)^2
\leq 4\int|\nabla u|^2dx\cdot\int|\nabla u_t|^2dx\\
&\leq 4  c_0  \varepsilon^2 \int |\nabla u_t|^2dx
\leq 4  c_0  \varepsilon^2  \sqrt{m}  E(t)
\leq c_1\varepsilon^2E(t),
\end{aligned}
\eeqs
with $c_1:= 4 c_0 \sqrt{m_1}$.

Therefore, for $c_2:=\max_{{} \, 0 \le s \le c_0 \varepsilon_0^2} |\frac 14 \frac{d}{ds}  ( \frac{1}{\sqrt {m}}) |$ and
\beqs
\label{12}
F(t):=E(t)- \frac 14 \frac{d}{ds}  \left( \frac{1}{\sqrt {m}}\right )  s'(t)^2,
\eeqs
we have that
\beqsn
E(t)-c_2 c_1\varepsilon^2E(t)\leq F(t)\leq E(t)+c_2 c_1\varepsilon^2E(t)
\eeqsn
and hence
\beqs
\label{9}
\frac{E(t)}{2}  \leq  F(t)\leq   \frac{3 E(t)}{2}
\eeqs
provided $0 \le \varepsilon \le \delta$, with $\delta = \min \big\{ \varepsilon_0 , 
1/ \sqrt{2 c_2 c_1}\big \} $.
Similarly, from \eqref{11} and \eqref{9} we get
\beqs
\label{cor}
\left|
\frac 14 \frac{d^2}{ds^2}  \left( \frac{1}{\sqrt {m}}\right ) s'(t)^3
\right | \leq 2   c_4  \varepsilon^3F(t)^{3/2} \quad  (0\leq\varepsilon\leq\delta)
\eeqs
for a suitable constant $c_4 > 0$.
From \eqref{4} we thus obtain the differential inequality
\beqs
\label{diffin}
F'(t)\leq  2 c_4  \varepsilon^3F(t)^{3/2}.
\eeqs
Since for $0\le \varepsilon \le \delta$ we have also  $F(0) \le  c_5  \varepsilon^2$, with
$c_5 > 0$, integrating \eqref{diffin}  we finally have
\beqsn
\sqrt{F(t)} \leq  \frac{\sqrt{F(0)}}{1-c_4  \varepsilon^3 \sqrt{F(0)}\,t }
 \leq \frac{\sqrt{c_5}\,\varepsilon}{1-c_4 \sqrt{c_5} \varepsilon^4 t} \quad \text{for}
\quad 0\le t < \frac{1} {c_4  \sqrt{c_5}\,\varepsilon^4} .
\eeqsn
This means that,  for $0 < \varepsilon \le \delta$,
\beqsn
F(t) \le 4 c_5  \varepsilon^2 \quad \text{if} \quad 0\le  t  \le  \frac{1} {2 c_4  \sqrt{c_5}\varepsilon^4}
=: 
\frac C{\varepsilon^4}.
\eeqsn
By \cite{AP} problem \eqref{CP} has a unique solution on
$\R^n \times[0,T_\varepsilon)$, with $T_\varepsilon > C /\varepsilon^4$.
\end{proof}

By the same computations we also have:
\begin{Cor} If we further assume that, for some $\alpha> 0$,
\beqs
\label{420}
\frac {d^2}{ds^2} \left ( \frac 1{\sqrt{m}} \right ) = O(s^\alpha) \quad {as} \quad s \rightarrow 0^+,
\eeqs
then we obtain the estimate: $T_\varepsilon > C/ \varepsilon^{4+2\alpha}$ when  $0 <  \varepsilon \le \delta $.
\end{Cor}
\begin{proof}
In fact \eqref{420}, together with \eqref{summ}, implies \eqref{cor}
with $\varepsilon^{3+ 2\alpha}$ instead of $\varepsilon^3$.
\end{proof}

\section{A condition for boundedness of the energy $E(t)$}

Sufficient conditions for boundedness of the energy $E(t)$ are already known.
Assuming, for instance, that $m$ is a $C^1$ function with $m(s)>0$ and coercive at
$\infty$, we have that 
\beqs
\label{b0}
s(t) \le
M^{-1} \left [ \int |u_t(x,0) |^2  dx +  M \left (\int |\nabla u(x,0) |^2  dx \right ) \right ] := \bar s,
\eeqs
and hence the condition that
$
V(t):= \int_0^t |s'(\tau)|d\tau
$
 is bounded for $t\in[0,T)$ is sufficient for boundedness of $E(t)$ because \eqref{stella}
implies
\beqs
\label{fiore}
E'(t) \le \sqrt{m} \left | \frac d{ds} \left (\frac 1{\sqrt{m}}\right )\right |  |s'(t) |   E(t),
\eeqs
which yields
\beqs  
\label{CS}
E(t)  \le  E(0) e^{c V(t)} \quad \text{for} \quad t \in [0,T)
\eeqs
for $c:=\max_{0\leq s\leq\bar s}|\sqrt{m}\frac{d}{ds}(\frac{1}{\sqrt{m}})|$.

Here, applying the second order identity \eqref{4}, we find a  {\it necessary} and {\it sufficient} condition for the boundedness of
 $E(t)$. More precisely, let us consider the Cauchy problem
\beqs
\label{CP2}
\begin{cases}
\ds u_{tt}-m \left(\int |\nabla u|^2dx\right)\Delta u=0,&(x,t)\in \R^n \times\R^+_t\cr
\ds u(x,0)=\phi_0(x), \quad
\ds u_t(x,0)=   \phi_1(x) \phantom{\int}
\end{cases}
\eeqs
with
$m\in C^2(\R^+)$, $m(s) >0 $, $m$ coercive at $\infty$, and  initial data $\phi_0\in H^2(\R^n),\phi_1 \in H^1(\R^n)$.

If $u\in C^k([0,T);H^{2-k}(\R^n))$ $(k=0,1,2)$ is the local solution of the Cauchy
problem \eqref{CP2} in $\R^n \times [0,T)$, $T > 0$, then:
\begin{Th} \label{th5}
The energy $E(t)$ of the solution $u $ is bounded in $[0,T)$ if and only if
\beqsn
S(t):=  \int_0^t \frac{d^2}{ds^2} \left ( \frac 1{\sqrt{m}}\right ) s'(\tau)^3 d\tau
\eeqsn
is bounded in  $[0,T)$.
\end{Th}

\begin{proof} 
 From \eqref{b0} we have, for $0 \le s \le \bar s$,
\beqsn
0 < \Lambda_0 \le m (s) \le \Lambda_1 \quad \text{and} \quad \left |\frac {dm}{ds}(s) \right | \le \Lambda_2,
\eeqsn
 for suitable constants $\Lambda_0$,  $\Lambda_1$, $ \Lambda_2$.
 Besides, from
the  second  order identity \eqref{4} we get
\beqs
\label{4biss}
\left [E - \frac{1}{4}  \frac {d}{ds}  \left ( \frac {1}{\sqrt {m}} \right)  {s'}^2\right]_0^t= -
\frac{1}{4} \int_0^t  \frac {d^2}{ds^2}  \left ( \frac {1}{\sqrt {m}} \right)   {s'(\tau) }^3 d\tau.
\eeqs

Now, let  us assume $E(t)$ bounded in $[0,T)$. Then, arguing as in \eqref{11},
 \beqs
 \label{2stelle}
 s'(t) ^2  \le   4  \bar s \sqrt{\Lambda_1}  E(t) \quad\text{and}\quad
 \left |\frac {d}{ds}  \left ( \frac {1}{\sqrt {m}} \right) (s(t)) \right | \le \frac {\Lambda_2}{2\Lambda_0^{3/2}}
 \quad \text{for}\quad t\in [0,T),
 \eeqs
 and hence from \eqref{4biss} we immediately deduce that  $S(t)$ remains bounded in $[0,T)$.

Conversely, let us suppose $S(t)$ bounded in $ [0,T)$ and let $ \bar t\in (0,T)$. If $ s'( \bar t  ) =0$, it is enough to observe that from
\eqref{4biss} we have:
\beqsn
E(\bar t ) = E(0) - \frac{1}{4}  \left [\frac {d}{ds}  \left ( \frac {1}{\sqrt {m}} \right)  {s'}^2\right]_{t=0}
- \frac{ S(\bar t  )}4.
\eeqsn
Otherwise, let us suppose $s'(\bar t )> 0 $ (the arguments will be similar if $s'(\bar t )< 0 $). Then, we define
\beqsn
t_0  =  \min \Big \{  t \in [ 0,  \bar t \, )\,  \Big |   \, 
 s'(\tau) > 0 \ \text{for}\   t < \tau \le \bar t  \Big \}.
\eeqsn
Clearly, at least one of the following must hold:
\beqsn
t_0= 0 \quad \text{or}\quad s'(t_0)=0.
\eeqsn
If $ t_0= 0$, having $ s'(t) > 0 $ in $ (  0, \bar t \, ]$,   we find
from \eqref{fiore}
\beqsn
E(\bar t )  \le   E(0) e^{c \int_0^{\bar t} |s'| d\tau}  =  E(0)  e^{c [s(\bar t  )-s(0)]}
\le E(0)  e^{c  \bar s}.
\eeqsn
where $c$ is constant introduced after \eqref{CS}.
\\
If $ s'(t_0) = 0$, we have
\beqsn
E(\bar t  )  \le  E(t_0) e^{c  \int_{t_0}^{\bar t} |s'|  d\tau } \le E(t_0)
 e^{c [s(\bar t  )-s(t_0)]}
\le E(t_0) e^{c\bar s} ,
\eeqsn
with
\beqsn
E(t_0) = E(0) - \frac{1}{4}  \left [\frac {d}{ds} \left ( \frac {1}{\sqrt {m}} \right)  {s'}^2\right]_{t=0}
  -\frac{S(t_0)} 4
\eeqsn
because of \eqref{4biss}.

In conclusion, in all cases $E(\bar t )$ is bounded by a constant  which depends only on the initial data and
the values of $ S(t)$ in $ [0,\bar t \, ]$.  
\end{proof}

\begin{Rem}
\begin{em}
Note that if $\frac{dm}{ds}(s)\geq0$ for all $s\in\R^+$ then formula \eqref{4biss} is enough to prove that $E(t)$ is bounded if and only if $S(t)$ is bounded, because of \eqref{2stelle}.

Moreover, as in Theorem~\ref{th}, if the initial data are sufficiently small we can avoid the assumption that $m$ is coercive at $\infty$.
Also, if $\left.u_t\right|_{t=0}\equiv0$, i.e. $\phi_1\equiv0$, then
\beqsn
M\left(\int|\nabla\phi_0|^2dx\right)<\int_0^{+\infty}m(h)dh
\eeqsn
and hence, arguing similarly as in \eqref{bound0}, we have that $s(t)$ is bounded and again coercivity at $\infty$ can be avoided.
\end{em}
\end{Rem}

\vspace*{3mm}
{\bf Acknowledgments.}
The first author is member of the Gruppo Nazionale per l'Analisi Ma\-te\-ma\-ti\-ca, la
Probabilit\`a e le loro Applicazioni (GNAMPA) of the Instituto Nazionale di Alta Matematica (INdAM).

\end{document}